\documentclass[12pt]{article}

\usepackage{amsmath,amssymb,amstext,amsthm} % Lots of math symbols and environments

\newenvironment{claimproof}[1][Proof of Claim]
{ \begin{proof}[#1]}
{\end{proof}  }

\theoremstyle{plain}
\newtheorem{theorem}{Theorem}[section]
\newtheorem{lemma}[theorem]{Lemma}
\newtheorem{proposition}[theorem]{Proposition}
\newtheorem{corollary}[theorem]{Corollary}

\theoremstyle{definition}
\newtheorem*{definition}{Definition}

\theoremstyle{remark}
\newtheorem*{remark}{Remark}
\newtheorem*{claim}{Claim}

\title{Cops and Robber Game with a Fast Robber \\ on Expander Graphs and Random Graphs}

\author{Abbas Mehrabian\thanks{\texttt{amehrabi@uwaterloo.ca}} \\ {\small University of Waterloo}}
\date{}

\begin{document}

\maketitle

\begin{abstract}
We consider a variant of the Cops and Robber game,
in which the robber has unbounded speed,
i.e.~can take any path from her vertex in her turn,
but she is not allowed to pass through a vertex occupied by a cop.
%We study this game on interval graphs, chordal graphs, planar graphs, and hypercube graphs.
Let $c_{\infty}(G)$ denote the number of cops needed to capture the robber in a graph $G$
in this variant.
We characterize graphs $G$ with $c_{\infty}(G)=1$,
and give an $O(|V(G)|^2)$ algorithm for their detection.
We prove a lower bound for $c_{\infty}$ of expander graphs,
and use it to prove three things.
The first is that if $np \geq 4.2\log n$ then the random graph $G=\mathcal G(n,p)$
asymptotically almost surely has ${\eta_1}/{p} \leq c_{\infty}(G) \leq {\eta_2 \log (np)}/{p}$,
for suitable constants $\eta_1$ and $\eta_2$.
The second is that a fixed-degree random regular graph $G$ with $n$ vertices
asymptotically almost surely has $c_{\infty}(G)= \Theta(n)$.
The third is that if $G$ is a Cartesian product of
$m$ paths, then $n/4km^2 \leq c_{\infty}(G) \leq n / k$,
where $n=|V(G)|$ and
$k$ is the number of vertices of the longest path.
%We show that if $G$ is an interval graph, then  $c_{\infty}(G) = O(\sqrt {|V(G)|})$,
%and we give a  polynomial time 3-approximation algorithm for finding $c_{\infty}(G)$ in interval graphs.
%We prove that for every $n$
%there exists an $n$-vertex chordal graph $G$ with $c_{\infty}(G) = \Omega(n / \log n)$.
%Let $tw(G)$ and $\Delta(G)$ denote the treewidth and the maximum degree of $G$, respectively.
%We prove that for every $G$,
%$tw(G) + 1 \leq (\Delta(G) + 1) c_{\infty}(G)$.
%Using this lower bound for $c_{\infty}(G)$, we show that
%(1) If $G$ is a planar graph
%(or more generally, if $G$ does not have a fixed apex graph as a minor),
%then $c_{\infty}(G) = \Theta(tw(G))$.
%This immediately leads to an $O(1)$-approximation algorithm for computing the $c_{\infty}$ of planar graphs.
%(2) If $G$ is the $m$-hypercube graph, then
%there exist constants $\eta_1,\eta_2>0$ such that
%$\eta_1 2^m / (m\sqrt m) \leq c_{\infty}(G) \leq \eta_2 2^m / m$.
\end{abstract}

\section{Introduction}
\label{chp:introduction}
The game of \emph{Cops and Robber} is a perfect information game,
played in a graph $G$.
The players are a set of cops and a robber.
Initially, the cops are placed at vertices
of their choice in $G$ (where more than one cop can be placed at a vertex).
Then the robber, being
fully aware of the cops' placement, positions herself at one of the vertices of $G$.
Then the cops and
the robber move in alternate rounds, with the cops moving first;
however, players are permitted to
remain stationary in their turn if they wish.
The players use the edges of $G$ to move from vertex to vertex.
The cops win, and the game ends, if eventually a cop moves to the vertex currently occupied
by the robber; otherwise, i.e.~if the robber can elude the cops forever, the robber wins.

This game was defined (for one cop) by Winkler and Nowakowski~\cite{game_definition_1}
and Quilliot~\cite{game_definition_2}, and has been studied extensively.
For a survey of results on this game, see the survey by Hahn~\cite{survey_hahn}.
The famous open question in this area is Meyniel's conjecture, published by Frankl~\cite{large_girth},
which states that for every connected graph on $n$ vertices, $O(\sqrt n)$ cops are sufficient to capture the robber.
The best result so far is that
$$n2^{-\left(1-o(1)\right)\sqrt{\log_2 n}}$$
cops are sufficient to capture the robber.
This was proved independently by Lu and Peng~\cite{lu_peng},
and Scott and Sudakov~\cite{scott_sudakov}.

One interesting fact about the Cops and Robber game is that
many scholars have studied the game,
and yet it is not really well understood:
although the upper bound  $O(\sqrt n)$
was conjectured in 1987,
%cops suffice to capture the robber in a connected graph on $n$ vertices,
no upper bound better than $n^{1-o(1)}$ has been proved since then.
As an another example, no efficient approximation algorithm
for finding the number of cops needed to capture the robber in a given graph has been developed.

One might try to change the rules of the game a little in order to get a more approachable problem,
and/or to understand what property of the original game causes the difficulty.
Several variations of the game have been studied,
by changing the rules slightly,
e.g.~by limiting the visibility of the cops~\cite{regular_visible_robber},
by limiting the visibility of both players~\cite{randomized_local_visibility},
by changing the definition of capturing~\cite{shooting_cop},
or by allowing the players to move only in a certain direction along each edge~\cite{variations}.

The approach chosen %in this paper is
by Fomin, Golovach, Kratochv\'{\i}l, Nisse, and Suchan~\cite{fast_robber_first_journal} is
to allow the robber move faster than the cops.
Inspired by their work,
in this paper we let the robber take \emph{any path} from her
current position in her turn,
but she is not allowed to pass through a vertex occupied by a cop.
The parameter of interest is the \emph{cop number} of $G$,
which is defined as the minimum number of cops needed to ensure that the cops can win.
We denote the cop number of $G$ by $c_{\infty}(G)$,
in which the $\infty$ at the subscript indicates that the robber has unbounded speed.
A nice fact about this variation is its analogy with
%the following:
%if we change the rules to make this a real-time game,
the so-called Helicopter Cops and Robber game
(defined in~\cite{helicopter}).
This is a real-time pursuit-evasion game with a robber of unbounded speed,
for which Seymour and Thomas have shown that the number of cops needed equals the treewidth
of the graph (which is a fairly well understood parameter)
plus one~\cite{helicopter}.
%Thus one may hope to get good bounds for the cop number in terms of treewidth by
%relating our variant of the Cops and Robber game
%and the Helicopter Cops and Robber game,
%and this is what we do in Section~\ref{chp:treewidth}.
However, one should not be deceived by this analogy;
the cop number can be arbitrarily smaller than the treewidth:
any graph with small domination number and large treewidth
(e.g.,~a complete graph) is such an example.
%Therefore, this paper can also be regarded as an attempt to find connections between
%the original Cops and Robber game
%and the Helicopter Cops and Robber game
%by studying an in-between game.
%Nevertheless, treewidth is closely connected with cop number,
%and in Section~\ref{chp:treewidth}, which is completely devoted to this connection,
%we prove bounds for cop number in terms of treewidth.
%In Sections~\ref{chp:planar},~\ref{chp:hypercube}~and~\ref{chp:large},
%we will see three applications of these bounds.
%Tree and path decompositions arise naturally and are important
%when studying the cop number, and
%the idea of several proofs in Sections~\ref{chp:interval}~and~\ref{chp:chordal}
%is based on them
%(although they do not appear explicitly in these sections).

The $c_{\infty}$ variant was first studied
by Fomin, Golovach, Kratochv{\'{\i}}l~\cite{fast_robber_first}.
They proved that computing $c_{\infty}(G)$ is an NP-hard problem, even if $G$
is a split graph.
(A \emph{split graph} is a graph whose vertex set can be partitioned into a clique and an independent set.)
This variant was further studied by Frieze, Krivelevich and Loh~\cite{variations},
%where the authors' approach is based on expansion.
%In~\cite{variations}, it is 
who showed that for each $n$, there exists a connected graph on $n$ vertices with cop number $\Theta(n)$.
As demonstrated in~\cite{variations}, expansion properties of a graph are closely connected with its cop number.
In this paper we further study this connection.
We obtain some lower bounds for the cop number of a graph in terms of its isoperimetric numbers (Section~\ref{chp:expander}).
Then we use these results to give lower bounds for the cop number of random graphs (Section~\ref{chp:random})
and for the cop number of  Cartesian products of graphs (Section~\ref{chp:product}).
In Section~\ref{chp:copwin}, we give a characterization of graphs $G$ with $c_{\infty}(G) = 1$,
and provide an $O(|V(G)|^2)$ algorithm for deciding if $G$ has this property.
%\item[Chapter~\ref{chp:lower_bound}:] Fix $s\in\mathbb{N}$. We prove that for every $n$ there exists $G$ with
%$$c_s(G) = \Omega \left(n^{{s}/({s+1})}\right).$$
%This result appears in~\cite{meyniel_generalize}.
%See~\cite{fast_robber_lower_bounds} for a simpler proof when $s=2,4$.
%Frieze~et~al.~\cite{variations} had proved that for every $n$ there exists $G$ with
%$c_s(G) = \Omega \left(n^{(s-3)/(s-2)}\right)$,
%and had asked whether there exist graphs $G$ with $c_2(G) = \omega(\sqrt n)$.
%This result improves their bound and gives a positive answer to their question,
%as we provide graphs with $c_2(G) = \Omega(n^{2/3})$.
%The best known general upper bound~\cite{variations} is not better than $c_s(G) \leq n^{1-o(1)}$.
%
%\item[Chapter~\ref{chp:treewidth}:] Let $tw(G)$ and $\Delta(G)$ denote the treewidth and maximum degree of $G$, respectively.
%We prove that for every $G$,
%$$\frac{tw(G) + 1 }{\Delta(G) + 1} \leq c_{\infty}(G) \leq tw(G) + 1,$$
%and provide examples for which these bounds are tight.
%
%\item[Chapter~\ref{chp:interval}:] If $G$ is an interval graph, then we prove that $c_{\infty}(G) = O(\sqrt n)$
%and provide examples for which this bound is tight.
%We also give a  polynomial time 3-approximation algorithm for finding $c_{\infty}(G)$.
%
%\item[Chapter~\ref{chp:chordal}:] We prove that for every $n$
%there exists a chordal graph $G$ with $c_{\infty}(G) = \Omega(n / \log n)$.
Let $G$ be a connected graph on $n$ vertices with maximum degree $\Delta$.
Let $\iota_e(G)$ and $\iota_v(G)$ denote the edge-isoperimetric and vertex-isoperimetric numbers of $G$, respectively.
In Section~\ref{chp:expander} we prove that for every $G$,
$$c_{\infty}(G) \geq \max \left\{ \frac{\iota_e n}{2\Delta^2}, \frac{\iota_v n}{4\Delta} \right\}.$$
In the subsequent two sections we give some applications of this result.
In Section~\ref{chp:random} we show that if $np \geq 4.2 \log n$, then
asymptotically almost surely the random graph $G = \mathcal G(n,p)$ has
$$c_{\infty}(G) = \Omega\left(\frac{n}{\Delta}\right) = \Omega\left(1/p\right).$$
If also $p=1-\Omega(1)$, then we prove that asymptotically almost surely $G$ has
$$c_{\infty}(G) = O(\log (np) / p).$$
In Section~\ref{chp:random} we also show that for every fixed $d$,
asymptotically almost surely a randomly chosen labelled $d$-regular graph $G$ on $n$ vertices has
$$c_{\infty}(G) =\Theta(n).$$
Let $P_n$ and $C_n$ denote a path and a cycle with $n$ vertices, respectively.
In Section~\ref{chp:product} we prove that if $G$ is the Cartesian product of $P_{n_1},P_{n_2},\dots,P_{n_m}$,
where $n_1 = \max\{n_i : 1\leq i\leq m\}$, then
$$\frac{n}{4n_1 m^2} \leq c_{\infty}(G) \leq \frac{n}{n_1}.$$
Moreover, if $G$ is the Cartesian product of $C_{n_1},C_{n_2},\dots,C_{n_m}$,
where $n_1 = \max\{n_i : 1\leq i\leq m\}$ is even, then
$$\frac{n}{2n_1 m^2} \leq c_{\infty}(G) \leq \frac{2n}{n_1}.$$
In Section~\ref{chp:samespeed} we briefly discuss a variation in which the cops and the robber have the same speed,
and we conclude with some open problems in Section~\ref{chp:future}.

\subsection{Preliminaries and notation}
Let $G$ be the graph in which the game is played.
In this paper $G$ is always finite,
and $n$  always denote the number of vertices of $G$.
Write $\delta$ and $\Delta$ for the minimum and maximum degree of $G$.
We will assume that $G$ is simple, because deleting multiple edges or loops
does not affect the set of possible moves of the players.
We consider only connected graphs, since the cop number of a disconnected graph
obviously equals the sum of the cop numbers for each connected component.
As we are only interested in studying the cop number,
we may assume without loss of generality
that the cops choose vertices of our choice in the beginning,
since they can move to the vertices of their choice later.

For a subset $A$ of vertices, the \emph{neighbourhood} of $A$, written $N(A)$,
is the set of vertices that have a neighbour in $A$,
and the \emph{closed neighbourhood} of $A$, written $\overline{N}(A)$, is the union $A \cup N(A)$.
If $A=\{v\}$ then we may write $N(v)$ and $\overline{N}(v)$ instead of $N(A)$ and $\overline{N}(A)$, respectively.
A \emph{dominating set} is a subset $A$ of vertices with $V(G) = \overline{N}(A)$,
and the \emph{domination number} of $G$ is the minimum size of a dominating set of $G$.
The subgraph induced by $A$ is written $G[A]$, and the subgraph induced by $V(G) - A$ is written $G-A$.

\section{Characterization of Graphs with Cop Number One}
\label{chp:copwin}

For the original Cops and Robber game,
graphs in which a single cop can capture the robber
have been characterized independently
by Nowakowski and Winkler~\cite{game_definition_1}
and by Quilliot~\cite{game_definition_2}.
In this section we characterize graphs $G$ with $c_{\infty}(G)=1$,
and give an $O(n^2)$ algorithm for their detection.

\begin{definition}[block, block tree]
Let $G$ be a connected graph.
By a \emph{block} of $G$, we mean either a maximal 2-connected subgraph of $G$,
or an edge of $G$ that is not contained in any 2-connected subgraph.
We may associate with $G$ a bipartite graph $B(G)$ with bipartition $(\mathcal B, S)$,
where $\mathcal B$ is the set of blocks of $G$ and $S$ is the set of cut vertices of $G$,
a block $B$ and a cut vertex $v$ being adjacent in $B(G)$ if and only if $B$ contains $v$.
The graph $B(G)$ is a tree, called the \emph{block tree} of $G$ (see for example~\cite{graph_theory}, page 121).
\end{definition}

\begin{lemma}
\label{lem:blocks_domination_one}
If $c_{\infty}(G)=1$ then every block of $G$ has domination number one.
\end{lemma}

\begin{proof}
Suppose for the sake of contradiction that $c_{\infty}(G)=1$ and $B$ is a block of $G$ with domination number larger than one.
So $B$ is a 2-connected subgraph.
Assume that there is a single cop in the game.
We claim that the robber can play in such a way that, at the end of each round,
if the cop is at a vertex $v$, then the robber is at a vertex $r \in V(B) \setminus \overline{N}(v)$.
This shows that she can elude the cop forever, which contradicts the assumption $c_{\infty}(G)=1$.

Assume that the cop starts at $v_0 \in V(G)$.
Since $B$ has domination number larger than one, there exists $r_0 \in V(B) \setminus \overline{N}(v_0)$.
The robber starts at $r_0$.
For every positive $i$, suppose that in round $i$, the cop moves to $v_i$.
Since $B$ has domination number larger than one, there exists $r_i \in V(B) \setminus \overline{N}(v_i)$.
As $B$ is 2-connected, there are two disjoint $(r_{i-1},r_i)$-paths in $G$,
so there exists an $(r_{i-1},r_i)$-path in $G$ that does not contain $v_i$.
The robber has unbounded speed and moves along that path to $r_i$, and the proof is complete.
\end{proof}

\begin{definition} [directed hole, hallway]
Let $u$ be a cut vertex of $G$, and $B$ be a block of $G$ containing $u$.
If $\{u\}$ is not a dominating set for $B$, then the pair $(B,u)$ is called a \emph{directed hole}.
Let $B,B'$ be two distinct blocks of $G$, and $B u_1 \dots u_k B'$ be the unique $(B,B')$-path in $B(G)$.
If both $(B,u_1)$ and $(B',u_k)$ are directed holes, then the pair $\{B, B'\}$ is called a \emph{hallway}.
\end{definition}

Note that if a block $B$ is not 2-connected,
then it consists of a single edge, and each of its vertices makes
a dominating set.
Hence, if $\{B,B'\}$ is a hallway, then both $B$ and $B'$ are maximal 2-connected subgraphs.
We will prove that a graph $G$ has $c_{\infty}(G) = 1$
if and only if each of its blocks has domination number one,
and it does not have a hallway.

\begin{lemma}
\label{lem:no_hallway}
If $c_{\infty}(G) = 1$, then $G$ does not have a hallway.
\end{lemma}

\begin{proof}
Suppose for the sake of contradiction that $c_{\infty}(G)=1$ and $\{B,B'\}$ is a hallway.
By the discussion above, $B$ and $B'$ are maximal 2-connected subgraphs.
Let $B u_1 \dots u_k B'$ be the unique $(B,B')$-path in $B(G)$.
Assume that there is a single cop in the game.
Since $(B,u_1)$ is a directed hole, there exists $b\in V(B) \setminus \overline{N}(u_1)$.
Similarly, since $(B',u_k)$ is a directed hole, there exists $b'\in V(B') \setminus \overline{N}(u_k)$.
Note that the distance between $b$ and $u_1$ in $G$ is at least 2,
and the distance between $u_k$ and $b'$ in $G$ is at least 2,
so the distance between $b$ and $b'$ in $G$ is at least 4.
We claim that the robber can play in such a way that, at the end of each round,
if the cop is at a vertex $v$, then she is at a vertex $r \in \{b,b'\} \setminus \overline{N}(v)$.
This shows that she can elude the cop forever, which contradicts the assumption $c_{\infty}(G)=1$.

Assume that the cop starts at $v_0 \in V(G)$.
As the distance between $b$ and $b'$ in $G$ is at least 4, there exists $r_0 \in \{b,b'\} \setminus \overline{N}(v_0)$ and the robber starts at $r_0$.
For every positive $i$, suppose that in round $i$, the cop moves to $v_i$.
At the end of round $i-1$, the robber was either at $b$ or at $b'$,
and by symmetry we may assume that she was at $b$.
If $b\notin \overline{N}(v_i)$, then the robber remains at $b$.
Otherwise $b\in \overline{N}(v_i)$ so
$v_i \neq u$ since $b\notin\overline{N}(u_1)$,
and $b'\notin \overline{N}(v_i)$ since the distance between $b$ and $b'$
in $G$ is at least 4.
There exists two disjoint $(b, u_1)$-paths, thus at least one of them is cop-free.
There is also a cop-free $(u_1,u_k)$-path and a cop-free $(u_k,b')$-path so the robber can move to $b'$ in her turn.
\end{proof}

The two above lemmas prove the ``only if'' part of the result we are going to prove.
For the other direction, we need another definition and a lemma.

\begin{definition}[end block]
Let $G$ be a connected graph such that $B(G)$ has more than one vertex.
The blocks of $G$ which correspond to leaves of $B(G)$
are referred to as its \emph{end blocks}.
\end{definition}

\begin{lemma}
\label{contracting_end_blocks}
Let $B$ be an end block of graph $G$, and $u$ be the unique cut vertex of $G$ contained in $B$.
Assume that $\{u\}$ is a dominating set for $B$.
Let $H$ be the graph obtained by contracting the subgraph $B$ into vertex $u$.
Then we have $c_{\infty}(H) \geq c_{\infty}(G)$.
\end{lemma}

\begin{proof}
We need to show that for every positive $c$, if $c$ cops can capture the robber in $H$,
then $c$ cops can capture the robber in $G$.
Assume that $c$ cops have a capturing strategy in $H$.
They may use the following strategy in $G$:
whenever the robber is at a vertex $r \in V(H)$, they move according to their strategy in $H$,
and when the robber moves to a vertex in $r\in V(G) \setminus V(H)$, they just ``imagine'' that the robber
is at $u$, and again  move according to their strategy in $H$.
Since the cops' strategy in $H$ is winning, they eventually will either capture the robber in $H$,
or capture the ``imagined'' robber at $u$.
In the former case, the robber is captured in $G$ as well.
In the latter case, there would be a cop at $u$ and the robber would be in $V(G) \setminus V(H)$.
Now, that cop can capture the robber in the next move, as $\{u\}$ is a dominating set for $B$, and $V(G) \setminus V(H) \subseteq V(B)$.
\end{proof}

We are now ready to prove the main result of this section.

\begin{theorem}
A connected graph $G$ has $c_{\infty}(G) = 1$
if and only if each of its blocks has domination number one,
and it does not have a hallway.
\end{theorem}

\begin{proof}
If $c_{\infty}(G) = 1$ then by Lemma~\ref{lem:blocks_domination_one}  each of the blocks of $G$ has domination number one,
and by Lemma~\ref{lem:no_hallway}, $G$ does not have a hallway.

Conversely, let $G$ be a connected graph such that each of its blocks has domination number one,
and it does not have a hallway.
We perform the following operation on $G$:
let $B$ be an arbitrary end block of $G$, and $u$ be the unique cut vertex of $G$ contained in $B$.
If $\{u\}$ is a dominating set for $B$, then we contract the subgraph $B$ into vertex $u$.
We repeat this operation until no such end block exists.
Let $H$ be the resulting graph.
Note that each of the blocks of $H$ is also a block of $G$.

\begin{claim}
The graph $H$ has a single block.
\end{claim}
\begin{claimproof}
If $H$ has more than one block, then since $B(H)$ is a tree, it has at least two leaves.
Let $B$ and $B'$ be two end blocks of $H$, $u$ and $u'$ be the unique cut vertices of $H$ with $u \in V(B)$ and $u' \in V(B')$.
Since we cannot perform the above operation on $H$, we know that $\{u\}$ is not a dominating set for $B$,
and $\{u'\}$ is not a dominating set for $B'$.
But then $\{B,B'\}$ would be a hallway in $G$, contradiction!
\end{claimproof}

Each block of $H$ is also a block of $G$, hence $H$ has domination number one, thus $c_{\infty}(H) = 1$.
Lemma~\ref{contracting_end_blocks} gives $c_{\infty}(G)\leq c_{\infty}(H)$,
and the proof is complete.
\end{proof}

We gave a mathematical characterization for graphs $G$ with $c_{\infty}(G)=1$.
Using this we give a simple algorithm for detecting such graphs.

\begin{corollary}
Let $G$ be a connected graph on $n$ vertices.
There exists an $O(n^2)$ algorithm to decide whether $c_{\infty}(G) = 1$.
\end{corollary}

\begin{proof}
The block tree of $G$ can be built in time $O(|E(G)|)$ using depth-first search (see for example~\cite{graph_theory}, page 142).
If block $B$ has $m$ vertices, then it is possible to find in time $O(m^2)$
all vertices $u\in V(B)$ such that $\{u\}$ is a dominating set for $B$ (using exhaustive search).
Hence in time $O(n^2)$ one can determine if all blocks of $G$ have domination number one,
and also find all directed holes $(B,u)$.
Using a simultaneous depth-first search on $B(G)$ starting from all the directed holes,
it is possible to decide if there is a hallway in $G$ in time $O(|E(B(G))|) = O(n)$.
Hence the total running time of the algorithm is $O(n^2)$.
\end{proof}

%%%%%%%%%%%%%%%%%%%%%%%%%%%%%%%%%%%%%%%%%%%%%%%%%%%%%%%%%%%%%%%%%%%%%%%%%%%%%%%%%%%%%%%%%%%%%%%%%%%%%%%%%%%%%%%%%%%%

\section[Expander Graphs] {Lower Bounds for Expander Graphs}
\label{chp:expander}

\begin{definition}[edge-isoperimetric number, vertex-isoperimetric number]
Let $G$ be a graph.
For a subset $S$ of vertices of $G$, write $\partial S$ for the set of edges with exactly one endpoint in $S$.
Define the \emph{edge-isoperimetric} and \emph{vertex-isoperimetric} numbers of $G$ as
$$\iota_e(G) = \min_{|S| \leq n/2} \frac {|\partial S|}{|S|},$$
$$\iota_v(G) = \min_{|S| \leq n/2} \frac {|{N}(S) \setminus S|}{|S|}.$$
\end{definition}

Note that for any graph $G$ we have $\iota_e(G) \leq \Delta$ (by taking $S$ to be any single vertex)
and $\iota_v(G) \leq 1$ (by taking $S$ to be any subset with $n/2$ vertices).

%The idea of using expansion properties of graphs to bound the cop numbers,
%first appeared in~\cite{variations}, where the authors proved the following.
%
%\begin{theorem}[\cite{variations}]
%Let $s\in\mathbb{N}$ and $\alpha = 1 + 1/s$.
%There is a function $p = p(n) = \alpha^{-(1-o(1)) \sqrt{\log_{\alpha} n}}$ for which the following holds.
%In every graph $G$ on $n$ vertices with $\Delta(G) < 1/p$ and $\iota_v(G) \geq p$,
%if the robber has fixed speed $s$, then $(1+o(1)) 2pn \sqrt{\log_{\alpha} n}$
%cops can capture her.
%\end{theorem}
%
%Using the above theorem, they proved that
%in a connected graph on $n$ vertices,
%$$\alpha^{-(1-o(1)) \sqrt{\log_{\alpha} n}}$$
%cops can capture a robber who has fixed speed $s$,
%where $s\in\mathbb{N}$ and $\alpha = 1 + 1/s$.

In this section we prove that for every graph $G$, we have
$$c_{\infty} (G)   \geq \frac{\iota_e n}{\Delta^2 - \Delta + \iota_e(\Delta+1)} \geq \frac{\iota_e n}{2\Delta^2},$$
and
$$c_{\infty}(G) \geq \max\left\{\frac{\iota_v n}{3\Delta + \iota_v(\Delta+1)}, \frac{\iota_v n}{4\Delta}\right\}.$$

\begin{lemma}
\label{lem:large_components}
Let $m$ be a positive integer such that for every subset $S$ of at most $m$ vertices,
$G-\overline{N}(S)$ has a connected component of size larger than $n/2$. Then $c_{\infty}(G) > m$.
\end{lemma}

\begin{proof}
Assume that there are $m$ cops in the game, and we give an escaping strategy for the robber.
The strategy has the following invariant: at the end of each round,
if the cops are positioned in a subset $S$ of vertices,
then the robber is at a vertex of the unique component of $G-\overline{N}(S)$ that has size larger than $n/2$.
Let $S_0$ be the subset of vertices that the cops occupy when the game starts.
By hypothesis, $G-\overline{N}(S_0)$ has a connected component $C_0$ of size larger than $n/2$,
and the robber starts at an arbitrary vertex of $C_0$.

Suppose that at the end of round $i$, the cops are in $S_i$,
and the robber is in a component $C_i$ of $G-\overline N (S_i)$ of size larger than $n/2$.
In round $i+1$, the cops move to a new set $S_{i+1} \subseteq \overline N(S_i)$, so the robber is not captured.
Let $C_{i+1}$ be the connected component of $G-\overline{N}(S_{i+1})$ that has size larger than $n/2$.
As both $C_i$ and $C_{i+1}$ have size larger than $n/2$, they intersect. Let $v \in C_i \cap C_{i+1}$.
Since $C_i$ is disjoint from $\overline{N}(S_i)$, at this moment there is no cop in $C_i$.
Moreover, $C_i$ is connected and the robber is in $C_i$, so she can move to $v$ in this round.
Hence at the end of round $i+1$, the robber is in $C_{i+1}$,
the connected component of $G-\overline{N}(S_{i+1})$ of size larger than $n/2$,
and the proof is complete.
\end{proof}

\begin{remark}
The idea in the proof was first used in~\cite{variations} to prove the existence of graphs with large cop number.
\end{remark}

Before proving the main result of this section, we need a technical lemma.
The proof is easy and we omit it.

\begin{lemma}
\label{lem:subset_sums}
Let $n,t$ be positive integers with $t \leq n$.
Let $a_1, a_2, \dots, a_m$ be positive integers such that each of them is at most $n/2$,
and their sum is $t$. Then we have the following.
\begin{itemize}
\item[(a)]
One can choose a subset of $\{a_1,\dots,a_m\}$ whose sum is between $t/3$ and $n/2$ (inclusive).
\item[(b)]
If $t\geq n/4$ then one can choose a subset of $\{a_1,\dots,a_m\}$ whose sum is between $n/4$ and $n/2$ (inclusive).
\end{itemize}

\end{lemma}

%\begin{proof}
%We may assume without loss of generality that
%$$a_1\geq a_2 \geq \dots \geq a_m.$$
%(a)
%We use induction on $m$.
%If $m\leq 3$, then $a_1$ is between $t/3$ and $n/2$, so we may assume that $m \geq 4$.
%Since $a_m + a_{m-1} \leq a_{m-2} + a_{m-3}$ and the sum of the $a_i$'s is $t$, which is not more than $n$,
%we have $a_m + a_{m-1} \leq n/2$.
%Define
%$$b_1 = a_1, b_2 = a_2, \dots, b_{m-2} = a_{m-2}, b_{m-1} = a_{m-1} + a_m.$$
%Then each of the $b_i$'s is at most $n/2$, and their sum is $t$.
%Thus by the induction hypothesis, there exists a subset of them whose sum is between $t/3$ and $n/2$.
%This gives a subset of the $a_i$'s with the same sum, and the proof is complete.
%(b)
%We use induction on $m$.
%If $m=1$ then we have $a_1 = t \geq n/4$ so $a_1$ is between $n/4$ and $n/2$, and we are done.
%So, we may assume that $m \geq 2$.
%If $a_{m-1}$ is at least $n/4$, then $a_{m-1}$ is between $n/4$ and $n/2$, and we are done.
%So, we may assume that $a_{m-1} < n/4$, thus $a_{m-1} + a_m < n/2$.
%Define
%$$b_1 = a_1, b_2 = a_2, \dots, b_{m-2} = a_{m-2}, b_{m-1} = a_{m-1} + a_m.$$
%Then each of the $b_i$'s is at most $n/2$, and their sum is $t$.
%Thus by induction hypothesis, there exists a subset of them whose sum is between $n/4$ and $n/2$.
%This gives a subset of the $a_i$'s with the same sum, and the proof is complete.\qedhere
%\end{proof}

Now we are ready to prove the main result of this section.

\begin{theorem}
\label{thm:isoperimetric_bounds}
For every graph $G$ we have
\begin{itemize}
\item[(a)]$\displaystyle\ c_{\infty} (G) \geq \frac{\iota_e n}{\Delta^2 - \Delta + \iota_e(\Delta+1)} \geq \frac{\iota_e n}{2\Delta^2}$,

\item[(b)]$\displaystyle\ c_{\infty} (G) \geq \frac{\iota_v n}{3\Delta + \iota_v(\Delta+1)}$,

\item[(c)]$\displaystyle\ c_{\infty} (G) \geq \frac{\iota_v n}{4\Delta}.$
\end{itemize}
\end{theorem}

\begin{proof}
Let $c=c_{\infty}(G)$.
By Lemma~\ref{lem:large_components} there exists a subset $S$ of at most $c$ vertices
such that $G-\overline{N}(S)$ has no component of size larger than $n/2$. We have
$$|\overline{N}(S)| \leq c(\Delta+1),\qquad |\overline {N}(S) \setminus S| \leq c\Delta, \quad\mathrm{and}\qquad
|\partial \overline{N}(S)| \leq c\Delta(\Delta-1),$$
where the last inequality holds since at most $c\Delta$ vertices of $\overline{N}(S)$ have a neighbour out of $\overline{N}(S)$,
and each has at most $\Delta-1$ such neighbours.
Let $T = V(G) \setminus \overline{N}(S)$, and let $A_1,A_2,\dots,A_m$ be the connected components of $G[T]$.
As $G[T]$ has no component of size larger than $n/2$, we have $|A_i| \leq n/2$ for all $i$.

(a) Since all of the $|A_i|$'s are at most $n/2$, for all $1\leq i \leq m$ we have $|\partial A_i| \geq \iota_e |A_i|$.
Thus
$$|\partial T| = \sum_{i=1}^{m} |\partial A_i| \geq \sum_{i=1}^{m} \iota_e |A_i| = \iota_e \sum_{i=1}^{m} |A_i| = \iota_e |T|.$$
This gives
\begin{align*}
c\Delta(\Delta-1) \geq  |\partial \overline{N}(S)| = |\partial T| \geq \iota_e |T| = \iota_e (n-|\overline{N}(S)|) \geq \iota_e(n-c(\Delta+1)).
\end{align*}
Part (a) now results by simplifying and noting that  $\iota_e \leq \Delta$.

(b)
By Lemma~\ref{lem:subset_sums} part (a), one can pick some components of $G[T]$
such that their union, $T'$, has size at least $|T|/3$ and at most $n/2$.
Then the set ${N}(T') \setminus T'$ has size at least $\iota_v |T'|$ and at most $|\overline N(S) \setminus S|$.
Thus
$$c\Delta \geq |\overline N(S) \setminus S| \geq \iota_v |T'| \geq \iota_v |T|/3 = \iota_v (n-|\overline{N}(S)|)/3 \geq \iota_v (n - c(\Delta+1)) / 3,$$
and part (b) follows after simplification.

(c)
If $|T| < n/4$, then we have $|\overline{N}(S)| > 3n/4$ so that $c(\Delta+1) > 3n/4$ and
$$c > \frac{3n}{4(\Delta + 1)} > \frac{\iota_v n}{4\Delta},$$
as $\iota_v \leq 1$ and $\Delta \geq 1$.

If $|T| \geq n/4$, then by Lemma~\ref{lem:subset_sums} part (b),
one can pick some components of $T$
such that their union has size at least $n/4$ and at most $n/2$.
Let $T'$ be their union.
Then the set ${N}(T') \setminus T'$ has size at least $\iota_v |T'|$ and at most $|\overline N(S) \setminus S|$, thus
$$c\Delta \geq |\overline{N}(S) \setminus S| \geq \iota_v |T'| \geq \iota_v n/4,$$
and part (c) follows.
\end{proof}

%Theorem~\ref{thm:isoperimetric_bounds} shows that every family of bounded-degree expanders have cop number $\Omega(n)$.
The existence of graph families with cop number $\Theta(n)$ has been proved by Frieze~et~al.~\cite{variations}.
However, their proof is nonconstructive.
A \emph{family of bounded-degree expanders} is a sequence $\{G_i\}_{i=1}^{\infty}$ of graphs,
where each $G_i$ has maximum degree $O(1)$ and vertex-isoperimetric number $\Omega(1)$.
Several constructions of families of bounded-degree expanders are known, see~\cite{expander_survey} for example.
Thus Theorem~\ref{thm:isoperimetric_bounds}, which 
shows that every family of bounded-degree expanders have cop number $\Theta(n)$,
enables one to construct graph families with cop number $\Theta(n)$.
This theorem also provides lower bounds for graphs with high expansion, for example random graphs (see Section~\ref{chp:random})
and Cartesian products of graphs (see Section~\ref{chp:product}).

%%%%%%%%%%%%%%%%%%%%%%%%%%%%%%%%%%%%%%%%%%%%%%%%%%%%%%%%%%%%%%%%%%%%%%%%%%%%%%%%%%%%%%%%%%%%%%%%%%%%%%%%%%%

\section[Random Graphs] {Bounds for Random Graphs}
\label{chp:random}

In this section we study $c_{\infty}(G)$ when $G$ is a random graph.
The original Cops and Robber game in random graphs has been studied by many authors, see for example~\cite{random_first,random_graphs_fixed_p,projective_plane_argument,zigzag}.
%\begin{definition}[$\mathcal{G}(n,p)$, asymptotically almost surely (a.a.s.)]
%Let $n$ be a positive integer and $p \in [0,1]$.
%The space $\mathcal{G}(n,p)$ is a probability space over all labelled graphs on $n$ vertices,
%such that for a randomly chosen $G\in \mathcal{G}(n,p)$ and a labelled graph $H$ on $n$ vertices,
%$$\mathbf{Pr}(G = H) = p^{|E(H)|}(1-p)^{\binom{n}{2} - |E(H)|}.$$
%This space may be viewed as $\binom{n}{2}$ independent coin flips, one for each pair of vertices,
%where the probability of drawing an edge between that pair is equal to $p$.
We denote an {E}rd\"{o}s-R\'{e}nyi random graph with parameters $n$ and $p$ by $\mathcal{G}(n,p)$.
%is a sample $G$ from the space $\mathcal{G}(n,p)$.
All asymptotics throughout are as $n\rightarrow \infty$.
We say that an event in a probability space holds \emph{asymptotically almost surely (a.a.s.)}
if the probability that it holds approaches 1 as $n$ goes to infinity.
All logarithms in this section are in base $e \approx 2.718$.
%Note that $p=p(n)$ can tend to zero as $n$ tends to infinity.
%\end{definition}
Let $\gamma(G)$ denote the domination number of $G$.

The main results of this section are the following.

\begin{itemize}
\item
Assume that $np \geq 4.2 \log n$. Then there exist positive constants $\eta_1, \eta_2$ such that a random graph $G= \mathcal G(n,p)$ a.a.s.~has
$$ \frac{\eta_1}{p} \leq c_{\infty}(G) \leq \frac{\eta_2 \log (np)}{p}.$$
%\item
%If $np = \omega(\sqrt{n \log n})$ and $p=1-\Omega(1)$, then a.a.s.
%$$c_{\infty}(G) = \gamma(G) = \Theta\left(\frac{\log n}{p}\right).$$
\item
Assume that $np = n^{\alpha + o(1)}$, where $1/2 < \alpha <1$. Then a.a.s.
$$c_{\infty}(G) = \Theta\left(\frac{\log n}{p}\right).$$
\item
If $np = n^{1-o(1)}$ and $p=1-\Omega(1)$, then a.a.s.
$$c_{\infty}(G) =(1+o(1)) \frac{\log n}{\log\frac{1}{1-p}}.$$
\item
Let $d \geq 3$ be fixed. Then a.a.s.~a randomly chosen labelled $d$-regular graph $G$ on $n$ vertices has
$$c_{\infty}(G) =\Theta(n).$$
\end{itemize}

We will use the following large deviation inequalities.
(See Corollary~A.1.10 and Theorem~A.1.13 in Appendix~A of~\cite{probabilistic_method}).

\begin{proposition}
\label{thm:chernoff}
Let $Y_1,\dots,Y_m$ be independent indicator random variables such that for all $i$, $\mathbb{E}[Y_i] = p = 1-q$.
Let $Y=Y_1+\dots+Y_m$ and $a>0$.
Then we have the following inequalities.
\begin{align*}
%Pr[Y - \mathbb{E}Y > a] & <  \exp(-2a^2/m). \\
\mathbf{Pr}[Y - \mathbb{E}Y < -a] & <  \exp\left[a - (a+qm) \log\left(1+\frac{a}{qm}\right)\right]. \\
\mathbf{Pr}[Y - \mathbb{E}Y < -a]& <  \exp(-a^2 / 2pm).
\end{align*}
\end{proposition}

%First we will prove a large deviation inequality.
%\begin{lemma}
%\label{lem:my_chernoff}
%Let $X_1,X_2,\dots,X_m$ be independent identically distributed indicator random variables with $\mathbb{E}X_i \geq q$ for all $1\leq i\leq m$.
%Then for any $0<b<1$,
%$$\mathbf{Pr}[X_1+\dots+X_m \leq bm] \leq  \left( 2 (1-q)^{1-b}\right)^{m}.$$
%\end{lemma}
%\begin{proof}
%Let $p = \mathbb{E}X_i \geq q$.
%We have
%\begin{align*}
%\mathbf{Pr}[X_1+\dots+X_m \leq bm] = & \sum_{i=0}^{\lfloor bm \rfloor} \binom{m}{i} p^i (1-p)^{m-i} \\
%\leq & \sum_{i=0}^{\lfloor bm \rfloor} \binom{m}{i} (1-p)^{m-i} \\
%\leq & \sum_{i=0}^{\lfloor bm \rfloor} \binom{m}{i} (1-p)^{m-mb} \\
%\leq & \, 2^m (1-p)^{m-mb} = \left(2(1-p)^{1-b}\right)^m \leq \left(2(1-q)^{1-b}\right)^m.\qedhere
%\end{align*}
%\end{proof}

Next we give a lower bound for vertex-isoperimetric number of random graphs,
which is of independent interest.
Such a bound does not seem to have appeared explicitly before.

\begin{theorem}
\label{thm:random_expansion}
Let $0<b<1$ be fixed.
Let $\beta = 1-b$ and let $t,k$ be constants such that
$$t > \frac{1 + \log 2}{\beta} - \log \beta, \qquad k > \frac{2t}{1-e^{-t}}.$$
If $np \geq k \log n$ then the random graph $G = \mathcal G(n,p)$ a.a.s.~has $\iota_v(G) \geq b$.
\end{theorem}

\begin{proof}
We show that the random graph $G = \mathcal{G}(n, 1-e^{-p})$ a.a.s.~has $\iota_v(G) \geq b$.
This proves the theorem, since $p \geq 1-e^{-p}$ and $\iota_v(G)$ does not decrease by adding edges to $G$.
Let $V(G) = \{v_1,\dots,v_n\}$.
For $1\leq r \leq n/2$, define
$$A^{(r)} = \{v_{n-r+1},\dots,v_n\},\quad X^{(r)} = |N(A^{(r)})|.$$
Note that $|A^{(r)}| =r$ and $X^{(r)} = X_1^{(r)} + \dots + X_{n-r}^{(r)}$,
where $X_i^{(r)}$ is the indicator random variable for $v_i \in N(A^{(r)})$.
For all $1\leq i\leq n-r$ we have
$$\mathbb{E}X_i^{(r)} = \mathbf{Pr} [v_i \in N(A^{(r)})] = 1 - e^{-pr}.$$
By symmetry (among the subsets of size $r$) and the union bound
%(which states that
%the probability that at least one of a certain set of events happen is at most the sum of their probabilities)
it suffices to prove that
$$\sum_{r=1}^{n/2} \binom{n}{r} \mathbf{Pr} \left[ X^{(r)} < br \right] = o(1).$$
We split this sum into two parts: $1 \leq r < t/p$ and $t/p \leq r \leq  n/2$.

First, let $t/p \leq r \leq n/2$.
Let $m=n-r$, $Y_i = X_{i}^{(r)}$ and $a=(n-r)(1-e^{-pr}) - br$.
The first inequality in Proposition~\ref{thm:chernoff} gives
$$\mathbf{Pr} \left[ X^{(r)} < br \right] < \exp\left[ (n-r)(1-e^{-pr})-br - (n-r-br) \log \left(e^{pr} - \frac{bre^{pr}}{n-r} \right)\right].$$
Recall that $\beta = 1 - b$.
Then $1 - \frac{br}{n-r} \geq 1 - b = \beta$ so that
$$n-r-br \geq \beta (n-r), \quad e^{pr} - \frac{bre^{pr}}{n-r} \geq \beta e^{pr}$$
Thus, we have
\begin{align*}
\sum_{r= t/p }^{n/2} \binom{n}{r} \mathbf{Pr} \left[ X^{(r)} < br \right] &
\leq 2^n \exp\left[ (n-r)(1-e^{-pr})-br - (n-r-br) (pr+\log \beta) \right] \\
& \leq \exp\left[ n \log 2 + (n-r)\left(1 - \beta pr - \beta \log \beta \right) \right] \\
& = \left (\exp\left[\log 2 + \frac{n-r}{n}\left(1- \beta pr -\beta \log \beta \right) \right]\right)^n.
\end{align*}
To show the latter is $o(1)$,
we need to show that $f_1(r) = \frac{n-r}{n}\left(\beta pr + \beta \log \beta - 1 \right) > \log 2$ if $t/p \leq r \leq n/2$.
The function $f_1(r)$ is concave on $[t/p, n/2]$ and hence achieves its minimum at an endpoint of this interval.

If $r=t/p$, then
$$f_1(r) = \left(1-\frac{t}{np}\right)\left(\beta t + \beta \log \beta - 1 \right).$$
Since $t$ was chosen so that $\beta (t + \log \beta) - 1 > \log 2$,
 and $np = \omega(1)$, $f_1(r) > \log 2$ for $n$ large enough.
If $r=n/2$, then
$$f_1(r) = \frac{1}{2} \left(\beta pn/2 + \beta\log \beta - 1 \right) = \omega(1).$$

%As $\mathbb{E}X_i^{(r)}  \geq 1 - \exp(-pr) \geq 1 - \exp(-t)$ for all $1\leq i \leq n/2$,
%Lemma~\ref{lem:my_chernoff} (with $m=n/2$ and $q = 1 - \exp(-t)$) gives
%$$\mathbf{Pr} [X^{(r)} < bn/2] \leq  \left(2 e^{-t(1-b)}\right) ^ {n/2}.$$
%Thus, as $\mathbf{Pr} \left[ X^{(r)} < br \right] \leq \mathbf{Pr} [X^{(r)} < bn/2]$, we have
%$$\sum_{r=\lceil t/p \rceil}^{n/2} \binom{n}{r} \mathbf{Pr} \left[ X^{(r)} < br \right]
%\leq  2^n \left(2 e^{-t(1-b)}\right) ^ {n/2} =   \left(8 e^{-t(1-b)}\right) ^ {n/2}$$
%which is $o(1)$ as $t$ satisfies $e^{t(1-b)} > 8$.

Now we handle the second part, $1\leq r < t/p$.
Let $m=n-r$, $Y_i = X_{i}^{(r)}$ and $a=(n-r)(1-e^{-pr}) - br$.
The second inequality in Proposition~\ref{thm:chernoff} gives
$$\mathbf{Pr} \left[ X^{(r)} < br \right] < \exp\left[ -\frac{\left((n-r)(1-e^{-pr})-br\right)^2}{2(n-r)(1-e^{-pr})} \right]
< \exp\left[br - \frac{(n-r)(1-e^{-pr})}{2} \right].$$
For any fixed $r$, $1\leq r < t/p$, we have
\begin{align*}
\binom{n}{r} \mathbf{Pr} \left[ X^{(r)} < br \right] & < \exp\left[r\log n + br - \frac{(n-r)(1-e^{-pr})}{2} \right].
\end{align*}
Therefore, to show that
$$\sum_{r=1}^{t/p} \binom{n}{r} \mathbf{Pr} \left[ X^{(r)} < br \right] = o(1),$$
it is enough to show that $\frac{(n-r)(1-e^{-pr})}{2} - r\log n - br = \Omega(n)$ if $1\leq r \leq t/p$.
The function $f_2(r) = \frac{(n-r)(1-e^{-pr})}{2} - r\log n - br$ is concave,
and achieves its minimum at its endpoints.

When $r=1$,
$$f_2(r) = (n-1)(1-e^{-p})/2 - \log n  -b = \Omega(n).$$
When $r=\frac{tn}{k \log n} \geq t/p$,
\begin{align*}
f_2(r) =
\frac{(n-r)(1-e^{-pr})}{2} - r\log n - br & \leq \frac{(n-\frac{tn}{k\log n})(1-e^{-t})}{2} - \frac{tn}{k\log n} (b + \log n) \\
& = n \left[\frac{(1-\frac{t}{k\log n})(1-e^{-t})}{2} - \frac{t}{k} - \frac{bt}{k\log n} \right],
\end{align*}
which is $\Omega(n)$ as $k$ was chosen such that
$$\frac{1-e^{-t}}{2} - \frac{t}{k} > 0.\qedhere$$
\end{proof}

For upper bounds, we will use some known bounds on the domination number $\gamma(G)$ of random graphs.
The following theorem has been proved in page 4 of the book by Alon and Spencer~\cite{probabilistic_method}.

\begin{theorem}[\cite{probabilistic_method}]
\label{thm:domination_upper_bound_original}
Every graph $G$ has
$$\gamma(G) \leq n\frac{1+\log(\delta+1)}{\delta+1}.$$
\end{theorem}

\begin{corollary}
\label{thm:domination_upper_bound_a}
If $np > 2\log n$ then a random graph $G=\mathcal{G}(n,p)$ a.a.s.~has
$$ \gamma(G) = O\left( \frac{n\log \delta}{\delta}\right)=O\left(\frac{\log (np)}{p}\right).$$
\end{corollary}

\begin{proof}
For $np > 2\log n$, a.a.s.~$\delta$ is $\Theta(np)$.
\end{proof}

The following theorem has been proved
by Bonato, Pra{\l}at, and Wang~\cite{random_graphs_fixed_p} when $p=o(1)$, and by
Wieland and Godbole~\cite{random_graphs_domination} when $p=\Omega(1)$.

\begin{theorem}[\cite{random_graphs_fixed_p,random_graphs_domination}]
\label{thm:domination_upper_bound_b}
If $p = 1 - \Omega(1)$, then
a random graph $G=\mathcal{G}(n,p)$ a.a.s~has
$$\gamma(G) \leq (1+o(1)) \frac{\log n}{\log\frac{1}{1-p}}.$$
%Note that if , then the right-hand-side is $\Theta (\frac{\log n}{p})$.
\end{theorem}

For a graph $G$, let $c_1(G)$ be the minimum number of cops that
can capture the robber in $G$, in the original Cops and Robber game
(in which the robber can move only to an adjacent vertex in her turn).
Then we have
$$c_1(G) \leq c_{\infty}(G) \leq \gamma(G).$$
The lower bound is obvious.
%: if the robber becomes faster,
%then the number of cops needed cannot decrease.
The upper bound is easy: if the cops start by occupying a dominating set,
they will capture the robber in the first round.

We are ready to prove the main theorem of this section,
which provides bounds for  cop numbers of the random graph $\mathcal G(n,p)$
for various ranges of $p$.

\begin{theorem}
\label{thm:random_bounds}
Let $G=\mathcal G(n,p)$. Then we have the following.
\begin{enumerate}
%\item[(b)]
%$$np = o(\log n), G\ \mathrm{connected}\ \Rightarrow\ \mathrm{a.a.s.}\ c_{\infty}(G) = \Omega\left(\frac{\delta n}{\Delta^2}\right).$$
\item[(a)]
If $np \geq 4.2 \log n$, then a.a.s.
\begin{align*}
c_{\infty}(G) & = \Omega\left(\frac{n}{\Delta}\right) = \Omega\left(\frac{1}{p}\right),\ \mathrm{and} \\
c_{\infty}(G) & = O\left(\frac{n \log \delta}{\delta}\right) = O\left(\frac{\log (np)}{p}\right).
\end{align*}
%\item[(d)]
%$$np = \omega\left(\sqrt{n \log n}\right) \Rightarrow\ \mathrm{a.a.s.}\ c_{\infty}(G) = \gamma(G).$$
\item[(b)]
If $np = n^{\alpha+o(1)}$ with $\frac{1}{2} < \alpha < 1$, then a.a.s.
$$c_{\infty}(G) = \Theta\left(\frac{\log n}{p}\right) = n^{1-\alpha + o(1)}.$$
\item[(c)]
If $np = n^{1-o(1)}$ and $p=1-\Omega(1)$, then a.a.s.
$$c_{\infty}(G) = (1+o(1)) \frac{\log n}{\log\frac{1}{1-p}}.$$
\end{enumerate}
\end{theorem}

%\begin{remark}
%Notice the gap between ranges of parts (a) and (b): when $np = \Omega(\log n)$ but $np < 4.2 \log n$,
%we do not give a lower bound for $c_{\infty}$.
%\end{remark}

\begin{proof}
\begin{enumerate}
%\item[(a)]
%Easy: if the cops start by occupying a dominating set, they will capture the robber in the first round.
%\item[(b)]
%It is known that if $\delta(G) = o(\log n)$, then
%a.a.s.~$\iota_e(G) = \delta(G)$~\cite{random_graph_isoperimetric}.
%Part (b) thus follows from part (a) of Theorem~\ref{thm:isoperimetric_bounds}.
\item[(a)]
Let $b=0.001, t=1.7$, and $k = 4.2$.
It follows from Theorem~\ref{thm:random_expansion}
that if $pn \geq k \log n$ then $G$ a.a.s.~has $\iota_v (G) \geq b$,
and the lower bound follows from part~(c) of Theorem~\ref{thm:isoperimetric_bounds},
and noting that in this range we have $\Delta = \Theta(np)$.
The upper bound follows from Corollary~\ref{thm:domination_upper_bound_a}.

%\item[(d)]
%Clearly $c_{\infty}(G) \leq \gamma(G)$.
%Since $np = \omega\left(\sqrt{n \log n}\right)$, we have $n\log n = o((np)^2) = o(\delta^2)$.
%By Corollary~\ref{thm:domination_upper_bound_a} we have
%$$\gamma(G) = O\left(\frac{n \log \delta}{\delta}\right) = O\left(\frac{n\log n}{\delta}\right) = o(\delta).$$
%So we may assume that
%$\gamma(G) < \delta(G)$.
%
%A.a.s.~$G$ is $\delta(G)$-connected so a.a.s.~it is $\gamma(G)$-connected.
%For proving that $\gamma(G) \leq c_{\infty}(G)$ a.a.s.,
%we need to show that at least $\gamma(G)$ cops are needed to capture a robber with unbounded speed
%if $G$ is a $\gamma(G)$-connected graph.
%Indeed if there are less than $\gamma(G)$ cops in the game,
%there exists a non-dominated vertex in every round,
%and there exists an unblocked path to that vertex since $G$ is $\gamma(G)$-connected,
%so the robber can move there, and will never be captured.

\item[(b)]
Bonato, Pra{\l}at, and Wang~\cite{random_graphs_fixed_p} proved that
if $np = n^{\alpha + o(1)}$, where $1/2 < \alpha < 1$,
then a.a.s.~in the original Cops and Robber game played in $G= \mathcal{G}(n,p)$,
$$c_1(G) = \Theta\left(\frac{\log n}{p}\right) = n^{1-\alpha + o(1)}.$$
%For other $s$, note that we have
%$$c_1(G) \leq c_2(G) \leq \dots \leq c_{\infty}(G) \leq \gamma(G).$$
On the other hand, by Corollary~\ref{thm:domination_upper_bound_a}, $\gamma(G)$ is a.a.s~at most
$$ O\left(\frac{\log (np)}{p}\right) = O\left(\frac{\log n}{p}\right)= n^{1-\alpha + o(1)}.$$
The result follows since $c_1(G) \leq c_{\infty}(G) \leq \gamma(G)$.

\item[(c)]
Bonato~et~al.~\cite{random_graphs_fixed_p} proved that
if $np=n^{1-o(1)}$  and $p=1-\Omega(1)$, then a.a.s.~in the original Cops and Robber game played in $G= \mathcal{G}(n,p)$,
$$c_1(G) = (1+o(1)) \frac{\log n}{\log \frac{1}{1-p}}.$$
%cops are needed to capture the robber.
%For other $s$, note that we have
%$$c_1(G) \leq c_2(G) \leq \dots \leq c_{\infty}(G) \leq \gamma(G).$$
On the other hand, by Theorem~\ref{thm:domination_upper_bound_b}, $\gamma(G)$ is a.a.s.~at most
$$ (1+o(1))  \frac{\log n}{\log \frac{1}{1-p}}.$$
The result follows since $c_1(G) \leq c_{\infty}(G) \leq \gamma(G)$.
\end{enumerate}
\end{proof}

Finally, we give bounds for $c_{\infty}$ of random regular graphs,
using the following theorem for their edge-isoperimetric number,
proved by Bollob\'{a}s~\cite{random_regular_isoperimetric}.

\begin{theorem}[\cite{random_regular_isoperimetric}]
Let $d\geq 3$ be fixed. Then
a.a.s.~a randomly chosen $d$-regular labelled graph $G$ on $n$ vertices has
$$\iota_e(G) \geq d/2 - \sqrt{d\log 2} - o(1).$$
\end{theorem}

\begin{corollary}
Let $d\geq 3$ be fixed. Then
a.a.s.~a randomly chosen $d$-regular labelled graph $G$ on $n$ vertices has
$$\frac{d - 2\sqrt{d\log 2}}{4d^2}\,n - o(n) \leq c_{\infty}(G) \leq \gamma(G) \leq \frac{1+\log(d+1)}{d+1}\,n.$$
\end{corollary}

\begin{proof}
The lower bound follows from the above bound for $\iota_e(G)$
and part (a) of Theorem~\ref{thm:isoperimetric_bounds}.
The upper bound for $\gamma(G)$ follows from Theorem~\ref{thm:domination_upper_bound_original}.
\end{proof}

%%%%%%%%%%%%%%%%%%%%%%%%%%%%%%%%%%%%%%%%%%%%%%%%%%%%%%%%%%%%%%%%%%%%%%%%%%%%%%%%%%%%%%%%%%%%%%%%%%%%%%%

\section[Cartesian Products] {Bounds for Cartesian Products of Graphs}
\label{chp:product}

Let $G_1,G_2,\dots,G_m$ be graphs.
Define $G$ to be the graph with vertex set $V(G_1)\times V(G_2)\times \dots \times V(G_m)$
with vertices $(u_1,u_2,\dots,u_m)$ and $(v_1,v_2,\dots,v_m)$ being adjacent if there exists an index $1\leq j \leq m$ such that
\begin{itemize}
\item $u_i = v_i$ for all $i\neq j$, and
\item $u_j$ and $v_j$ are adjacent in $G_j$.
\end{itemize}
Then $G$ is called the \emph{Cartesian product} of $G_1,G_2,\dots,G_m$.
%\end{definition}
%\begin{remark}
%If every $G_i$ is isomorphic to an edge, then the graph $G$ is called the \emph{$m$-hypercube} graph and denoted by $\mathcal {H}_m$.

Neufeld and Nowakowski~\cite{speed_one_products} have studied the original Cops and Robber game
played on products of graphs.
They have determined exactly the number of cops needed to capture the robber,
when $G$ is the Cartesian product of complete graphs,
and when $G$ is the Cartesian product of an arbitrary number of trees and cycles.
In this section we study $c_{\infty}(G)$ when $G$ is a Cartesian product of graphs.

%%\begin{definition}[Cartesian product]
%Let $G_1,G_2,\dots,G_m$ be graphs.
%Define $G$ to be the graph with vertex set $V(G_1)\times V(G_2)\times \dots \times V(G_m)$
%with vertices $(u_1,u_2,\dots,u_m)$ and $(v_1,v_2,\dots,v_m)$ being adjacent if there exists an index $1\leq j \leq m$ such that
%\begin{itemize}
%\item $u_i = v_i$ for all $i\neq j$, and
%\item $u_j$ and $v_j$ are adjacent in $G_j$.
%\end{itemize}
%Then $G$ is called the \emph{Cartesian product} of $G_1,G_2,\dots,G_m$.
%%\end{definition}
%
%\begin{remark}
%If every $G_i$ is isomorphic to an edge, then the graph $G$ is actually the $m$-dimensional hypercube,
%which is denoted by $\mathcal {H}_m$.
%\end{remark}

The following theorem is the main result of this section.

\begin{theorem}
\label{thm:bounds_products}
Let $G_1,G_2,\dots,G_m$ be graphs and let $n_i$ denote the number of vertices of $G_i$ for $1\leq i\leq m$.
Let $G$ be the Cartesian product of $G_1,G_2,\dots,G_m$, and $n = |V(G)| = n_1 n_2 \dots n_m$.
Let $\Delta_i$ be the maximum degree of $G_i$, for $1\leq i\leq m$.
Then we have
\begin{itemize}
\item[(a)]
$$\frac{\min\{\iota_e(G_i) : 1\leq i\leq m \}n}{4(\Delta_1+\dots+\Delta_m)^2} \leq c_{\infty}(G) \leq \frac{nc_{\infty}(G_1)}{n_1}.$$
Note that the upper bound holds for any ordering of the graphs.
%one can choose the graph with the smallest $c_{\infty}$ as $G_1$,
%in order to optimize the upper bound.
\item[(b)]
If every $G_i$ is a path and $n_1 = \max\{n_i : 1\leq i\leq m\}$, then
$$\frac{n}{4n_1 m^2} \leq c_{\infty}(G) \leq \frac{n}{n_1}.$$
\item[(c)]
If every $G_i$ is a cycle, $n_1 = \max\{n_i : 1\leq i\leq m\}$, and $n_1$ is even, then
$$\frac{n}{2n_1 m^2} \leq c_{\infty}(G) \leq \frac{2n}{n_1}.$$
%\item[(d)]
%There exist positive constants $\kappa_1,\kappa_2$
%such that
%if every $G_i$ is isomorphic to the complete graph on $k$ vertices, then
%$$\frac{\kappa_1 n}{m^{3/2} k} \leq c_{\infty}(G) \leq \min\left\{ \frac{n}{k},\frac{\kappa_2 n}{\sqrt m}\right\}.$$
%\item[(e)]
%If every $G_i$ is isomorphic to an edge, i.e.~if $G$ is the $m$-dimensional hypercube $\mathcal{H}_m$,
%then there exist constants $\eta_1,\eta_2>0$ such that
%$$\frac{\eta_1 n}{m\sqrt m} \leq c_{\infty}(G) \leq \frac{\eta_2 n}{m}.$$
\end{itemize}
\end{theorem}

\begin{remark}
When every $G_i$ is isomorphic to an edge, it has been proved~\cite{interval_chordal_planar} using other techniques
that there exist constants $\alpha_1,\alpha_2>0$ such that
$$\frac{\alpha_1 n}{m\sqrt m} \leq c_{\infty}(G) \leq \frac{\alpha_2 n}{m}.$$
\end{remark}

\begin{proof}

\begin{itemize}
\item[(a)]
Chung and Tetali~\cite{products_isoperimetric_chung} have proved that
$$\iota_e(G) \geq \min\{\iota_e(G_i) : 1\leq i\leq m\} / 2.$$
Noting that $\Delta = \Delta_1 + \dots + \Delta_m$,
the lower bound follows from part~(a) of Theorem~\ref{thm:isoperimetric_bounds}.

For the upper bound we give a strategy for $nc_{\infty}(G_1) / n_1$ cops to capture a robber in $G$.
Let $k = c_{\infty}(G_1)$.
We consider two games:
the first one, which we call the \emph{real game}, is a game with $nk / n_1$ cops played in $G$;
and the second one, the \emph{virtual game}, %is the usual Cops and Robber game with $k$ cops and a robber with unbounded speed.
is a game in which $k$ virtual cops are capturing a {virtual robber} in $G_1$.
Given a winning strategy for the cops in the virtual game, we give a capturing strategy for the cops in the real game.
We translate the moves of the cops from the virtual game to the real game, and translate the moves of the robber from the real game to the virtual game, in such a way that all the translated moves are valid, and if the robber is captured in the virtual game, then she is captured in the real game as well.
By definition, there is a winning strategy for the cops in the virtual game.
Hence, the real cops have a winning strategy in the real game.

%We consider a \emph{virtual game}, in which $k$ virtual cops are capturing a {virtual robber} in $G_1$.
%(Using a virtual game for bounding the cop number is also used in the proof of Lemma~\ref{lem:helicopter_link},
%where it has been explained in more detail.)
For every virtual cop, we put $n/n_1 = n_2 n_3 \dots n_m$ real cops in the real game,
such that if the virtual cop is in $u_1 \in V(G_1)$, then the real cops occupy $\{u_1\} \times V(G_2) \times \dots \times V(G_m)$.
Also, if the real robber is at $(v_1,\dots,v_m) \in G$, then the virtual robber is at $v_1\in G_1$.
It is not hard to see that the real cops can move in such a way that these constraints hold throughout the games.
Hence, once the virtual robber has been captured, the real robber has also been captured, and the proof is complete.

\item[(b)]
Azizo\u{g}lu and E\u{g}ecio\u{g}lu~\cite{grids_isoperimetric} have proved that
$$\iota_e(G) = \left\lfloor \frac{n_1}{2}\right\rfloor^{-1} \geq \frac{2}{n_1}.$$
As $G$ has $n$ vertices and maximum degree $2m$,
the lower bound follows from part (a) of Theorem~\ref{thm:isoperimetric_bounds}.
The upper bound follows from part (a) of the present theorem, since $G_1$ is a path and has $c_{\infty}(G_1) = 1$.

\item[(c)]
Azizo\u{g}lu and E\u{g}ecio\u{g}lu~\cite{cylinders_isoperimetric_number} have proved that
$$\iota_e(G) = \frac{4}{n_1}.$$
As $G$ has $n$ vertices and maximum degree $2m$,
the lower bound follows from part (a) of Theorem~\ref{thm:isoperimetric_bounds}.
The upper bound follows from part (a) of the present theorem, since $G_1$ is a cycle and  has $c_{\infty}(G_1) = 2$.
\qedhere
\end{itemize}

\end{proof}

%%%%%%%%%%%%%%%%%%%%%%%%%%%%%%%%%%%%%%%%%%%%%%%%%%%%%%%%%%%%%%%%%%%%%%%%%%%%%%%%%%%%%%%%%%%%%%%%%%%%%%%%%%%%

\section{The Same-Speed Variation}
\label{chp:samespeed}

In the concluding remarks of~\cite{variations} a variation is proposed in which all players have the same speed.
In this short section we prove that the cop number of a graph in this variation
equals  the cop number of a related graph in the original Cops and Robber game,
in which all players have speed one.

\begin{definition} [$c_{a,b}(G)$, $G_t$]
Let $a$ and $b$ be positive integers.
Let $c_{a,b}(G)$ denote the cop number of $G$ when the robber has speed $a$
and the cops have speed $b$.
%Note that in fact $c_s(G) = c_{s,1}(G)$.
That is, each cop can move along a path of length $b$ in his turn,
and the robber can move along a cop-free path of length $a$ in her turn.
Let $t$ be a positive integer, and let $G_{t}$ be the graph with vertex set $V(G)$
with $u,v\in V(G_t)$ being adjacent if their distance in $G$
is at most $t$.
\end{definition}

\begin{theorem}
For any graph $G$ and any positive integer $t$ we have
$$c_{t,t}(G) = c_{1,1}(G_t).$$
\end{theorem}

\begin{proof}
Consider the Cops and Robber game played in $G_t$ with both players having speed one. Call this game the \emph{original game},
and consider the Cops and Robber game played in $G$ with all players having speed $t$,  and call this game the \emph{alternative game}.
The set of possible moves for each player is almost the same in the two games,
the only difference is that there could be a possible move for the robber in the original game,
which is not possible in the alternative game:
if the robber is at $u$, and $v$ is a vertex at distance at most $t$ from $u$ (in $G$),
then she can always move from $u$ to $v$ in the original game, but,
in the alternative game, all of the $(u,v)$-paths of length at most $t$ may be blocked by a cop.

But, notice that if in some round of the original game, the robber moves from $u$ to $v$ in her turn,
such that there is a $(u,v)$-path of length at most $t$ in $G$ with a cop standing at one of its internal vertices,
then the robber will be captured in the next round.
This is because that condition implies that the cop's vertex is at distance at most $t$ from $v$ (in $G$),
hence he can capture the robber in the next round.
We deduce that such a move results in an immediate capture in the original game,
and the robber better not do it.
Apart from that kind of move, which we saw does not really give an advantage to the robber,
the set of moves for the players are the same in the two games, and the equality follows.
\end{proof}

%%%%%%%%%%%%%%%%%%%%%%%%%%%%%%%%%%%%%%%%%%%%%%%%%%%%%%%%%%%%%%%%%%%%%%%%%%%%%%%%%%%%%

\section{Open Problems}
\label{chp:future}

In this section we present a few open questions and research directions on this game.

\begin{enumerate}
%\item
%It is known that $f_1(n) = \Omega(\sqrt n)$.
%We proved that $f_s(n) = \Omega (n^{s/s+1})$ for every $s\in\mathbb{N}$ (see Theorem~\ref{thm:lower_bound}).
%Meyniel conjectured that $f_1(n) = \Theta(\sqrt n)$.
%Generalizing this conjecture, we conjecture that
%$$f_s(n) = \Theta (n^{s/s+1})$$
%for every $s\in\mathbb{N}$.
%In other words, we conjecture that for any $s\in\mathbb{N}$,
%just $O(n^{s/s+1})$ cops are enough to capture a robber having speed $s$,
%in a connected graph with $n$ vertices.
%This conjecture also appears in~\cite{fast_robber_lower_bounds}.
%
%This seems to be a difficult problem (even for the $s=1$ case the best
%known asymptotic bound is $f_1(n) \leq n^{1-o(1)}$, which is far from the
%conjectured $O(\sqrt n)$ bound).
%The best upper bound so far (for general $s$),
%given by Frieze~et~al.~\cite{variations}, is the following:
%$$\mathrm{If}\ \alpha = 1 + s^{-1}\mathrm{,\ then\ }f_s(G) \leq n \alpha^{-(1-o(1))\sqrt{\log_{\alpha}n}}.$$

%\item
%Fomin~et~al.~\cite{fast_robber_first_journal} asked about the complexity of computing $c_{\infty}(G)$ when $G$ is an interval graph.
%We proved that this problem is 3-approximable (see Theorem~\ref{thm:intervals_approximation}), but it is still not known if it is $\mathcal{NP}$-hard or not.
%
%\item
%We proved that there exist chordal graphs $G$ with $c_{\infty}(G) = \Omega(n / \log n)$ (see Theorem~\ref{chordal_lower}).
%Are there  chordal graphs $G$ with $c_{\infty}(G) = \Omega(n)$ ?

\item
When $np \geq 4.2 \log n$,
in part (a) of Theorem~\ref{thm:random_bounds},
we have a.a.s~determined the cop number of $\mathcal{G}(n,p)$
up to an $O(\log(np))$ factor.
Can one close this gap?

\item
In part (b) of Theorem~\ref{thm:bounds_products} we have determined
$c_{\infty}$ for the Cartesian product of $m$ paths,
up to an $O(m^2)$ factor. Can one close the gap?
%In part (e) of the same theorem, we have determined  $c_{\infty}$ for
%the $m$-dimensional hypercube (which can be considered as the Cartesian product of $m$ paths of length one)
%up to an $O(\sqrt m)$ factor. The same question can be asked here: what is the correct value?
%\item
%We have proved several bounds for $c_{\infty}$ in various classes of graphs.
%Can one find similar (perhaps weaker) bounds for $c_s$, $s$ fixed or $s<n$,
%by using the same ideas?

\item
Fomin, Golovach, and Kratochv\'{\i}l~\cite{fast_robber_first} proved that computing $c_{\infty}(G)$ is NP-hard.
Is this problem in NP?
To show that this problem is in NP, one needs to prove that there is always an efficient
way to describe the cops' strategy. %This has been done for the Helicopter Cops and Robber game~\cite{helicopter}.

\end{enumerate}

\noindent\textbf{Acknowledgement.}
The author is grateful to Nick~Wormald for continuous support and lots of fruitful discussions.

\bibliographystyle{amsplain}
\bibliography{graphsearching}

\end{document}